\documentclass[10pt,reqno]{amsart} 

\usepackage{amsthm,amsfonts,amscd,amsmath}
\usepackage[hidelinks]{hyperref} 
\usepackage{url}
\usepackage{xcolor}
\hypersetup{
    colorlinks,
    linkcolor={red!50!black},
    citecolor={blue!50!black},
    urlcolor={blue!80!black}
}

\usepackage[normalem]{ulem}
\usepackage{graphicx} 
\usepackage{datetime} 
\usepackage{arydshln} 
\usepackage{cancel}
\usepackage{caption,subcaption}
\numberwithin{figure}{section}
\numberwithin{table}{section}

\newcommand{\cf}{\textit{cf.\ }} 
\newcommand{\Iverson}[1]{\ensuremath{\left[#1\right]_{\delta}}} 

\title[Lambert Series Factorizations and Factor Pairs]{
       New Factor Pairs for Factorizations of Lambert Series Generating Functions
} 
\author[Mircea Merca and Maxie D. Schmidt]{
        Mircea Merca \\ 
        Academy of Romanian Scientists \\ 
        Splaiul Independentei 54, Bucharest, 050094 Romania \\ 
        \href{mailto:mircea.merca@profinfo.edu.ro}{mircea.merca@profinfo.edu.ro} \\ 
        \\ 
        Maxie D. Schmidt \\ 
        School of Mathematics \\ 
        Georgia Institute of Technology \\ 
        Atlanta, GA 30332 USA \\ 
        \href{mailto:maxieds@gmail.com}{maxieds@gmail.com}
} 
\date{\today} 

\allowdisplaybreaks 

\theoremstyle{plain} 
\newtheorem{theorem}{Theorem}

\newtheorem{cor}[theorem]{Corollary}
\numberwithin{theorem}{section}

\theoremstyle{remark} 
\newtheorem{example}[theorem]{Example}
\newtheorem{remark}[theorem]{Remark}

\begin{document} 

\begin{abstract} 
We prove several new variants of the Lambert series factorization theorem 
established in the first article 
``Generating special arithmetic functions by Lambert series factorizations'' 
by Merca and Schmidt (2017). 
Several characteristic examples of our new results are presented in the article to 
motivate the formulations of the generalized factorization theorems. 
Applications of these new factorization results include new identities involving the 
Euler partition function and the generalized sum-of-divisors functions, the M\"obius function, 
Euler's totient function, the Liouville lambda function, von Mangoldt's lambda function, and the 
Jordan totient function. 
\end{abstract}

\maketitle

\section{Introduction} 

\subsection{Lambert series factorization theorems} 

We consider recurrence relations and matrix equations 
related to \emph{Lambert series} expansions of the form 
\cite[\S 27.7]{NISTHB} \cite[\S 17.10]{HARDYANDWRIGHT} 
\begin{align}
\label{eqn_LambertSeriesfb_def} 
\sum_{n \geq 1} \frac{a_n q^n}{1-q^n} & = \sum_{m \geq 1} b_m q^m,\ |q| < 1, 
\end{align} 
for prescribed arithmetic functions 
$a: \mathbb{Z}^{+} \rightarrow \mathbb{C}$ and 
$b: \mathbb{Z}^{+} \rightarrow \mathbb{C}$ where $b_m = \sum_{d | m} a_d$. 
As in \cite{MERCA-SCHMIDT1}, we are interested in so-termed Lambert series 
factorizations of the form 
\begin{align} 
\label{eqn_FundLSFactorizationForm} 
\sum_{n \geq 1} \frac{a_n q^n}{1-q^n} & = \frac{1}{C(q)} \sum_{n \geq 1} \left( 
     \sum_{k=1}^n s_{n,k} a_k\right) q^n, 
\end{align} 
for arbitrary $\{a_n\}_{n \geq 1}$ and where specifying one of the sequences, 
$c_n := [q^n] 1 / C(q)$ or $s_{n,k}$ with $C(0) := 1$, 
uniquely determines the form of the other. 
In effect, we have ``\emph{factorization pairs}'' in the 
expansions of \eqref{eqn_FundLSFactorizationForm}. 
The special case of 
\[
(C(q), s_{n,k}) \equiv \left((q; q)_{\infty}, s_o(n, k)-s_e(n, k)\right), 
\] 
where $s_o(n, k)$ and $s_e(n, k)$ are respectively the number of $k$'s in all 
partitions of $n$ into an odd (even) number of distinct parts is considered in the 
references \cite{MERCA-SCHMIDT1,MERCA-LSFACTTHM,SCHMIDT-LSFACTTHM}. 
We generalize this result in two key new ways in the next sections. 

Central to the definition of our factorization pairs in 
\eqref{eqn_FundLSFactorizationForm} is the next matrix identity 
providing a factorized representation of special arithmetic functions 
generated by Lambert series expansions where 
\begin{align*} 
A_n & := \left(s_{i,j}\right)_{1 \leq i,j \leq n} 
     \quad\text{ and }\quad 
A_n^{-1} := \left(s_{i,j}^{(-1)}\right)_{1 \leq i,j \leq n}, 
\end{align*} 
and the one-dimensional sequence of $\{B_{m}\}_{m \geq 0}$ depends on the 
arithmetic function, $b_n$, implicit to the expansion of 
\eqref{eqn_LambertSeriesfb_def} and the factorization pair, $(C(q), s_{n,k})$. 
\begin{align} 
\label{eqn_fn_matrix_eqn}
\begin{bmatrix} a_1 \\ a_2 \\ \vdots \\ a_n \end{bmatrix} & = 
     A_n^{-1} \begin{bmatrix} B_0 \\ B_1 \\ \vdots \\ B_{n-1} \end{bmatrix} 
\end{align} 
Thus in order to construct a valid factorization pair we require that both 
the fundamental factorization result in 
\eqref{eqn_FundLSFactorizationForm} hold, and that the 
corresponding construction provides an identity of the form in 
\eqref{eqn_fn_matrix_eqn} for an application-dependent, 
suitable choice of the sequence, $B_m$ (see below). 

\subsection{Significance of our new results} 

In the article we prove several variants and properties of the Lambert series 
factorization theorem defined by \eqref{eqn_FundLSFactorizationForm}. 
Namely, in Section \ref{Section_NaturalGensOfFactorPairs} and 
Section \ref{Section_VariantsOfTheFactThms} we prove 
Theorem \ref{theorem_1}, Theorem \ref{theorem_MainThm_LSFactProps}, and then 
Theorem \ref{theorem_GenLSFactThm_vI} and Theorem \ref{theorem_GenLSFactThm_vII} 
which provide interesting generalized variations of the first two factorization 
theorem results. 
Each of these factorization theorems suggest new relations between sums of an arbitrary 
sequence $\{a_n\}_{n \geq 1}$ over the divisors of an integer $n$ as in 
\eqref{eqn_LambertSeriesfb_def} and more additive identities involving the same sequence. 
Our results proved in the article relate the two branches of additive and multiplicative 
number theory in many interesting new ways. 
Moreover, our new theorems connect several famous special multiplicative functions 
with divisor sums over partitions which are additive in nature. 

Even though there are a number of important results connection the theory of divisors with 
that of partitions and special classical partition functions, these results are more or less 
scattered in their approach. We propose to continue the study of the relationships between 
divisors and partitions with the goal of identifying common threads between these connections 
by the means of our unified factorization theorems of Lambert series generating functions. 
On the multiplicative number theory side, we connect the Euler partition function $p(n)$ 
with other important number theoretic functions including Euler's totient function, 
the M\"obius function, Liouville's lambda function, von Mangoldt's lambda function, the 
Jordan totient functions, and the generalized sum-of-divisors functions by extending the 
results first proved in \cite{MERCA-SCHMIDT1,MERCA-LSFACTTHM,SCHMIDT-LSFACTTHM}. 

\section{Natural generalizations of the factor pairs} 
\label{Section_NaturalGensOfFactorPairs} 

\begin{theorem}
	\label{theorem_1}
	Suppose that $C(q)$ in \eqref{eqn_FundLSFactorizationForm} is fixed. 
	Then for all integers $n, k \geq 1$, 
	we have that 
	\begin{align*} 
	\tag{i}
	s_{n,k} 
	     & = 
	     \sum_{i=1}^{\left\lfloor \frac{n}{k} \right\rfloor} [q^{n-i\cdot k}]C(q), 
	\end{align*} 
	i.e., so that we have a generating function for the general case of $s_{n,k}$ in the form of 
	\begin{align*} 
	\tag{ii}
	s_{n,k} & = [q^n] \frac{q^k}{1-q^k} C(q). 
	\end{align*} 
\end{theorem}

\begin{proof}
We rewrite \eqref{eqn_FundLSFactorizationForm} as
$$C(q) \sum_{k=1}^{\infty} \frac{q^k}{1-q^k} a_k = \sum_{k=1}^{\infty} \left( \sum_{n=1}^{\infty} s_{n,k} q^n\right) a_k.$$
Equating the coefficient of $a_k$ in this identity, gives
$$ \sum_{n=1}^{\infty} s_{n,k} q^n = \frac{q^k}{1-q^k} C(q).$$
Rewriting this relation
$$\sum_{n=1}^{\infty} s_{n,k} q^n = C(q) \sum_{n=1}^{\infty} q^{k\cdot n},$$ 
we derived the first claimed relation, which we note easily implies the second,  
where we have invoked the Cauchy multiplication of two power series.	
\end{proof}

\begin{remark} 
We remark that the general factorization in \eqref{eqn_FundLSFactorizationForm} can be easily derived considering the following identity
$$
\sum_{k=1}^n \left( \sum_{d|k} a_d \right) [q^{n-k}] C(q)
=\sum_{k=1}^{n} \left( \sum_{i\geq 1} [q^{n-i\cdot k}] C(q) \right)  a_k.
$$
The case of $C(q)\equiv (q;q)_\infty$ in Theorem \ref{theorem_1} can be rewritten considering Euler's pentagonal number theorem, i.e.,
$$(q;q)_\infty = \sum_{j=0}^{\infty} (-1)^{\lceil j/2 \rceil} q^{G_j},$$
where the exponent
	\begin{equation*}
	G_j = \frac{1}{2} \left\lceil \frac{j}{2} \right\rceil \left\lceil \frac{3j+1}{2} \right\rceil,\qquad j\geq 0,
	\end{equation*}
is the $j^{th}$ generalized pentagonal number. 
In particular, for $n,k>0$ we have that 
\begin{equation*}
s_o(n,k)-s_e(n,k) = \sum_{k | n-G_j} (-1)^{\lceil j/2 \rceil}, 
\end{equation*}
where the sum runs over all positive multiple of $k$ of the form $n-G_j$.
\end{remark}

Theorem \ref{theorem_1}  allows us to give another very interesting special case of 
\eqref{eqn_FundLSFactorizationForm} considered in \cite{MERCA-LSFACTTHM,MERCA-SCHMIDT1}.

\begin{cor}
\label{cor_spcase_of_theorem1_v1} 
	For arbitrary $\{a_n\}_{n \geq 1}$,
	$$\sum_{n \geq 1} \frac{a_n q^n}{1-q^n} = (q;q)_\infty \sum_{n \geq 1} \left( 
	\sum_{k=1}^n s_{n,k} a_k\right) q^n,
	$$
	where $s_{n,k}$ is the number of $k$'s in all unrestricted partitions of $n$.
\end{cor}
\begin{proof}
	We take into account the fact that
	$$\frac{q^k}{1-q^k} \cdot \frac{1}{(q;q)_\infty}$$
	is the generating function for the number of $k$'s in all unrestricted partitions of $n$.
	This generating function implies our result. 
\end{proof}

\begin{example}[Applications of the Corollary] 
The result in Corollary \ref{cor_spcase_of_theorem1_v1} 
allows us to derive many special case identities involving Euler's partition function and various arithmetic functions. 
More precisely, by the well-known famous special cases Lambert series identities 
expanded in the introduction to \cite{MERCA-SCHMIDT1}, for $n \geq 1$ we have that 
	\begin{align*}
	& \sum_{k=1}^{n} \sigma_x(k) p(n-k) = \sum_{k=1}^n k^x s_{n,k},\\
	&  p(n-1) = \sum_{k=1}^{n} \mu(k) s_{n,k}, \\
	& \sum_{k=1}^{n} k p(n-k) = \sum_{k=1}^n \phi(k) s_{n,k},\\
	& \sum_{k\geq 1} p(n-k^2) = \sum_{k=1}^n \lambda(k) s_{n,k},\\
	& \sum_{k=1}^n \Lambda(k)p(n-k) = \sum_{k=1}^n \log(k) s_{n,k},\\
	& \sum_{k=1}^n 2^{\omega(k)} p(n-k) = \sum_{k=1}^n | \mu(k) | s_{n,k},\\
	& \sum_{k=1}^n k^t p(n-k) = \sum_{k=1}^n J_t(k) s_{n,k},
	\end{align*}
	where $s_{n,k}$ is the number of $k$'s in all unrestricted partitions of $n$. 
	Moreover, in the case where $a_n \equiv 1$ in the corollary, for $n > 0$ we have that 
	\begin{align*}
	& \sum_{k=-\infty}^{\infty} (-1)^k S(n-k(3k+1)/2) = \sigma_0(n), 
	\end{align*} 
	and that 
	\begin{align*} 
	& \sum_{k=1}^{n} \sigma_0(k) p(n-k) = S(n),
	\end{align*}
	where $S(n)$ is number of parts in all partitions of $n$ (also, sum of largest parts of all partitions of $n$).
	Similarly, in the special case where $a_n := n$, for $n \geq 1$ we have that 
	\begin{align*}
	& \sum_{k=-\infty}^{\infty} (-1)^k (n-k(3k+1)/2)p(n-k(3k+1)/2) = \sigma_1(n), 
	\end{align*} 
	where 
	\begin{align*} 
	& \sum_{k=1}^{n} \sigma_1(k) p(n-k) = n \cdot p(n).
	\end{align*}
\end{example}

\begin{cor}[A Known Factorization] 
	For arbitrary $\{a_n\}_{n \geq 1}$, we have that 
	$$\sum_{n \geq 1} \frac{a_n q^n}{1-q^n} = (q^2;q)_\infty \sum_{n \geq 1} \left( p(n-1)a_1 + 
\sum_{k=2}^n s'_{n,k} a_k\right) q^n,
$$	
	where $s'_{n,k}$ is the number of $k$'s in all unrestricted partitions of $n$ that do not contain $1$ as a part. 
\end{cor}
\begin{proof} 
We consider \eqref{eqn_FundLSFactorizationForm} with $C(q)=(q^2;q)_{\infty}^{-1}$. According to Theorem \ref{theorem_1}, the generating function of $s'_{n,1}$ is given by
$$\frac{q}{1-q}\cdot \frac{1}{(q^2;q)_\infty} = \frac{q}{(q;q)_\infty} = \sum_{n=1}^{\infty} p(n-1) q^n.$$
For $k>1$, we see that the generating function of $s'_{n,k}$ is given by
$$\frac{q^k}{1-q^k} \cdot \frac{1}{(q^2;q)_\infty} = \frac{q^k}{1-q^k} \cdot \frac{1-q}{(q;q)_\infty},$$
which is the generating function for the number of $k$'s in all partitions of $n$ that do not contain $1$ as a part.
\end{proof} 

\begin{example}[More Applications of the Corollary] 
	We denote by $p_1(n)$ the number of partition of $n$ that do not contain $1$ as a part.
	For $n \geq 1$ and fixed $x \in \mathbb{C}$, we have that 
	\begin{align*}
	& \sum_{k=1}^{n} \sigma_x(k) p_1(n-k) = p(n-1)+\sum_{k=2}^n k^x s'_{n,k},\\
	&  -p(n-2) = \sum_{k=2}^{n} \mu(k) s'_{n,k}, \\
	& \sum_{k=1}^{n} k p_1(n-k) = p(n-1)+\sum_{k=2}^n \phi(k) s'_{n,k},\\
	& \sum_{k\geq 1}  p_1(n-k^2) = p(n-1)+\sum_{k=2}^n \lambda(k) s'_{n,k},\\
	& \sum_{k=1}^n \Lambda(k)  p_1(n-k) = \sum_{k=2}^n \log(k) s'_{n,k},\\
	& \sum_{k=1}^n 2^{\omega(k)}  p_1(n-k) = p(n-1)+\sum_{k=2}^n | \mu(k) | s'_{n,k},\\
	& \sum_{k=1}^n k^t  p_1(n-k) = p(n-1)+\sum_{k=2}^n J_t(k) s_{n,k}^{\prime},
	\end{align*}
	where $s'_{n,k}$ is the number of $k$'s in all partitions of $n$ that do not contain $1$ as a part.
\end{example}

\begin{cor}[Another Known Factorization] 
\label{cor_AnotherKnownFactorization}
	For arbitrary $\{a_n\}_{n \geq 1}$, we have that 
	$$\sum_{n \geq 1} \frac{a_n q^n}{1-q^n} = (q^3;q)_\infty \sum_{n \geq 1} \left( p_2(n-1)a_1 + p_1(n-2)a_2+ 
	\sum_{k=3}^n s''_{n,k} a_k\right) q^n,
	$$	
	where $p_k(n)$ is the number of partition of $n$ that do not contain $k$ as a part and 
	$s''_{n,k}$ is the number of $k$'s in all unrestricted partitions of $n$ that do not contain $1$ or $2$ as a part. 
\end{cor}

\begin{proof} 
	We consider \eqref{eqn_FundLSFactorizationForm} with $C(q)=(q^3;q)_{\infty}^{-1}$. According to Theorem \ref{theorem_1}, the generating function of $s''_{n,1}$ is given by
	$$\frac{q}{1-q}\cdot \frac{1}{(q^3;q)_\infty} = \frac{q(1-q^2)}{(q;q)_\infty} = \sum_{n=1}^{\infty} p_2(n-1) q^n.$$
	The generating function for  $s''_{n,2}$ is
	$$\frac{q^2}{1-q^2}\cdot \frac{1}{(q^3;q)_\infty} = \frac{q^2}{(q^2;q)_\infty} = \sum_{n=1}^{\infty} p_1(n-2) q^n.$$
	For $k>2$, we see that the generating function of $s''_{n,k}$ is given by
	$$\frac{q^k}{1-q^k} \cdot \frac{1}{(q^3;q)_\infty} = \frac{q^k}{1-q^k} \cdot \frac{(1-q)(1-q^2)}{(q;q)_\infty},$$
	which is the generating function for the number of $k$'s in all partitions of $n$ that do not contain $1$ or $2$ as a part.
\end{proof}

\begin{cor}[A Generalization of the Known Factorizations] 
For integers $m \geq 1$ and arbitrary $\{a_n\}_{n \geq 1}$, we have a Lambert series 
factorization given by 
\begin{align*} 
\sum_{n \geq 1} \frac{a_n q^n}{1-q^n} & = \sum_{n \geq 1} \left( 
     \sum_{i=1}^{m-1} \sum_{j=1}^{\lfloor n / i \rfloor} p_{m-1}(n-i \cdot j) a_i + 
     \sum_{k=m}^n s_{n,k}^{(m-1)} a_k\right) q^n, 
\end{align*} 
where $p_m(n)$ denotes the number of partitions of $n$ that do not contain 
$1, 2, \ldots, m$ as a part and where $s_{n,k}^{(m)}$ denotes the number of $k$'s in all 
unrestricted partitions that do not contain $1, 2, \ldots, m$ as a part. 
\end{cor} 
\begin{proof} 
The proof of Corollary \ref{cor_AnotherKnownFactorization} is the starting point for proving this 
generalized result. In particular, for the factorization pair determined by 
$C(q) := (q^m; q)_{\infty}^{-1}$ in \eqref{eqn_FundLSFactorizationForm}, 
we have that for $1 \leq i < m$ the coefficient on the right-hand-side of the factorization is 
given by 
\begin{align*} 
\frac{q^i}{1-q^i} \frac{(1-q)(1-q^2) \cdots (1-q^{m-1})}{(q; q)_{\infty}} & = 
     \sum_{n \geq 1} \left(\sum_{j=1}^{\lfloor n/i \rfloor} p_{m-1}(n-ij)\right) q^n. 
\end{align*} 
Similarly, by Theorem \ref{theorem_1} for $k \geq m$ 
we see that the right-hand-side coefficient of 
$a_k$ satisfies the following generating function over $n$: 
\begin{align*} 
\frac{q^k}{1-q^k} \cdot \frac{1}{(q^m; q)_{\infty}} & = \frac{q^k}{1-q^k} \cdot 
     \frac{(1-q)(1-q^2) \cdots (1-q^{m-1})}{(q; q)_{\infty}}. 
     \qedhere
\end{align*} 
\end{proof} 

\begin{theorem}[Generalized Factorization Theorem Identities] 
\label{theorem_MainThm_LSFactProps}
Suppose that the factorization pair $(c_n, s_{n,k})$ in 
\eqref{eqn_FundLSFactorizationForm} is fixed where $c_n := [q^n] 1 / C(q)$. 
Then for all integers $n, k \geq 1$ and $m \geq 0$ with $1 \leq k \leq n$, 
we have that 
\begin{align} 
\tag{i} 
s_{n,k}^{(-1)} & = \sum_{d|n} c_{d-k} \cdot \mu(n / d) \\ 
\tag{ii} 
c_{n-k} & = \sum_{d|n} s_{n,k}^{(-1)} \\ 
\tag{iii} 
B_m & = b_{m+1} + \sum_{k=1}^m [q^k] C(q) \cdot b_{m+1-k}. 
\end{align} 
\end{theorem} 
\begin{proof}[Proof of (i) and (ii)] 
This result is equivalent to showing that 
\[
c_{n-k} = \sum_{d|n} s_{d,k}^{(-1)}, 
\] 
which we do below by mimicking the proof from the reference 
\cite[\S 3]{MERCA-SCHMIDT1}. 
In particular, we consider the Lambert series over the sequence of 
$s_{n,k}^{(-1)}$ for a fixed integer $k \geq 1$ and note its 
factorization from \eqref{eqn_FundLSFactorizationForm} in the form of 
\begin{align*} 
\sum_{d|n} s_{d,k}^{(-1)} & = \sum_{m=0}^n \delta_{n-k,m} \cdot c_m = c_{n-k}. 
     \qedhere
\end{align*} 
\end{proof} 
\begin{proof}[Proof of (iii)] 
By the matrix representation of our factorization theorem given in 
\eqref{eqn_fn_matrix_eqn}, we see by a generating function argument 
starting from \eqref{eqn_FundLSFactorizationForm} that 
\begin{align*} 
B_{n-1} & = \sum_{k=1}^n s_{n,k} a_k \\ 
     & = 
     [q^n] C(q) \sum_{m \geq 1} b_m q^m \\ 
     & = 
     b_n + \sum_{k=1}^n [q^k] C(q) b_{n-k}, 
\end{align*} 
when $C(0) \equiv 1$ as in the factorization theorem stated in the introduction. 
\end{proof} 

Note that (i) in the proposition implies the following closed-form 
generating function for the Lambert series over the inverse matrix 
sequences by M\"obius inversion: 
\begin{align*} 
\sum_{n \geq 1} \frac{s_{n,k}^{(-1)} q^n}{1-q^n} & = \frac{q^k}{C(q)}. 
\end{align*} 
We have additional formulas that relate the sequences implicit to the 
choice of a fixed factorization pair in the form of 
\eqref{eqn_FundLSFactorizationForm}. 
Namely, we see that for $m \geq 1$ 
\begin{align*} 
b_m & = \sum_{d|m} a_d = \sum_{j=0}^m \sum_{k=1}^j s_{j,k} a_k c_{m-j}. 
\end{align*} 
We also have the following determinant-based recurrence relations proved as in the 
reference \cite[\S 2]{MERCA-SCHMIDT1} between the sequences, 
$s_{n,k}$ and $s_{n,k}^{(-1)}$, which are symmetric in that these 
identities still hold if one sequence is interchanged with the other: 
\begin{align} 
\label{eqn_snk_snkinv_det-based_recrels} 
s_{n,j}^{(-1)} & = - \sum_{k=1}^{n-j} s_{n,n+1-k}^{(-1)} \cdot s_{n+1-k,j} + 
     \delta_{n,j} \\ 
\notag 
     & = 
     - \sum_{k=1}^{n-j} s_{n,n-k} \cdot s_{n-k,j}^{(-1)} + \delta_{n,j} \\ 
\notag
     & = 
     -\sum_{k=1}^{n} s_{n,k-1} \cdot s_{k-1,j}^{(-1)} + \delta_{n,j}. 
\end{align} 

\begin{remark}[Tables of Special Matrix Entries]
The tables given in Appendix A on page \pageref{AppendixA_Tables_Figures} 
provide several concrete 
examples of the matrix sequences, $s_{n,k}$ and $s_{n,k}^{(-1)}$, implicit to 
specific factorizations in the form of \eqref{eqn_FundLSFactorizationForm}. 
In particular, these tables correspond to the matrix entries where the 
generating functions, $C(q)$, are defined respectively to be 
$C(q) := (-q; q)_{\infty}^{-1}, (a; q)_{\infty}^{-1}, (q; q^2)_{\infty}^{\pm 1}$. 
The listings in these tables provide explicit special cases that serve to 
demonstrate the results in Theorem \ref{theorem_MainThm_LSFactProps}. 
\end{remark} 


\section{Variations of the factorization theorems} 
\label{Section_VariantsOfTheFactThms}

\subsection{Motivation} 

One topic suggested by M. Merca as we considered generalizations of the factorization theorems 
both in this article and in our first article \cite{MERCA-SCHMIDT1} is to consider what happens in the 
form of Theorem \ref{theorem_MainThm_LSFactProps} part (i) when the 
M\"obius function is replaced by any other special 
multiplicative function, $\gamma(n)$, such as Euler's totient function, $\phi(n)$, or for example by 
von Mangoldt's function, $\Lambda(n)$. In its direct form, the factorization theorem in 
\eqref{eqn_FundLSFactorizationForm} does not accommodate a transformation of this form. 
However, if we change our specification of the fundamental factorization in the theorems from the 
previous section to allow the instance of $a_k$ in the left-hand-side sums of 
\eqref{eqn_FundLSFactorizationForm} to be a function, $\widetilde{a}_k$, 
depending on $\gamma(n)$ and the Lambert series sequence, $a_n$, 
we obtain several interesting new results. The next examples where 
$(C(q), \gamma(n)) := ((q; q)_{\infty}, \phi(n)), ((q; q)_{\infty}, n^{\alpha})$ for some fixed 
$\alpha \in \mathbb{C}$ provide the motivation for the statement of the more general theorem 
given in the next subsection. 

\begin{figure}[ht!]

\begin{minipage}{\linewidth} 
\begin{center} 
\tiny
\begin{equation*} 
\boxed{ 
\begin{array}{cccccccccccccccccc} \hline 
 1 & 0 & 0 & 0 & 0 & 0 & 0 & 0 & 0 & 0 & 0 & 0 \\
 -2 & 1 & 0 & 0 & 0 & 0 & 0 & 0 & 0 & 0 & 0 & 0 \\
 -2 & -1 & 1 & 0 & 0 & 0 & 0 & 0 & 0 & 0 & 0 & 0 \\
 2 & -2 & -1 & 1 & 0 & 0 & 0 & 0 & 0 & 0 & 0 & 0 \\
 -1 & 1 & -1 & -1 & 1 & 0 & 0 & 0 & 0 & 0 & 0 & 0 \\
 8 & -1 & -1 & -1 & -1 & 1 & 0 & 0 & 0 & 0 & 0 & 0 \\
 -5 & 3 & 1 & 0 & -1 & -1 & 1 & 0 & 0 & 0 & 0 & 0 \\
 2 & 1 & 2 & -1 & 0 & -1 & -1 & 1 & 0 & 0 & 0 & 0 \\
 3 & 1 & -2 & 2 & 0 & 0 & -1 & -1 & 1 & 0 & 0 & 0 \\
 1 & -3 & 3 & 1 & 0 & 0 & 0 & -1 & -1 & 1 & 0 & 0 \\
 -11 & 1 & 1 & 1 & 1 & 1 & 0 & 0 & -1 & -1 & 1 & 0 \\
 -2 & 6 & -1 & -2 & 2 & -1 & 1 & 0 & 0 & -1 & -1 & 1 \\
 \hline
\end{array}
}
\end{equation*}
\end{center} 
\subcaption*{(i) $s_{n,k}$} 
\end{minipage} 

\begin{minipage}{\linewidth} 
\begin{center} 
\tiny 
\begin{equation*} 
\boxed{ 
\begin{array}{cccccccccccccccccc} \hline 
 1 & 0 & 0 & 0 & 0 & 0 & 0 & 0 & 0 & 0 & 0 & 0 \\
 2 & 1 & 0 & 0 & 0 & 0 & 0 & 0 & 0 & 0 & 0 & 0 \\
 3 & 1 & 1 & 0 & 0 & 0 & 0 & 0 & 0 & 0 & 0 & 0 \\
 5 & 2 & 1 & 1 & 0 & 0 & 0 & 0 & 0 & 0 & 0 & 0 \\
 6 & 2 & 1 & 1 & 1 & 0 & 0 & 0 & 0 & 0 & 0 & 0 \\
 8 & 5 & 3 & 1 & 1 & 1 & 0 & 0 & 0 & 0 & 0 & 0 \\
 10 & 3 & 2 & 2 & 1 & 1 & 1 & 0 & 0 & 0 & 0 & 0 \\
 13 & 7 & 4 & 3 & 2 & 1 & 1 & 1 & 0 & 0 & 0 & 0 \\
 14 & 7 & 6 & 3 & 2 & 2 & 1 & 1 & 1 & 0 & 0 & 0 \\
 18 & 12 & 6 & 5 & 4 & 2 & 2 & 1 & 1 & 1 & 0 & 0 \\
 20 & 8 & 6 & 5 & 4 & 3 & 2 & 2 & 1 & 1 & 1 & 0 \\
 27 & 18 & 14 & 9 & 6 & 5 & 3 & 2 & 2 & 1 & 1 & 1 \\
 \hline 
\end{array}
} 
\end{equation*} 
\end{center} 
\subcaption*{(ii) $s_{n,k}^{(-1)}$} 
\end{minipage}

\caption{The generalized factorization where $(C(q), \gamma(n)) := ((q; q)_{\infty}, \phi(n))$} 
\label{figure_genfactpair_spcase_v1} 

\end{figure} 

\begin{figure}[ht!]

\begin{minipage}{\linewidth} 
\begin{center} 
\tiny
\begin{equation*} 
\boxed{ 
\begin{array}{cccccccccccccccccc} \hline 
 1 & 0 & 0 & 0 & 0 & 0 & 0 & 0 & 0 & 0 & 0 \\
 -1-2^{\alpha } & 1 & 0 & 0 & 0 & 0 & 0 & 0 & 0 & 0 & 0 \\
 -1+2^{\alpha }-3^{\alpha } & -1 & 1 & 0 & 0 & 0 & 0 & 0 & 0 & 0 & 0 \\
 2^{\alpha }+3^{\alpha } & -1-2^{\alpha } & -1 & 1 & 0 & 0 & 0 & 0 & 0 & 0 & 0 \\
 3^{\alpha }-5^{\alpha } & 2^{\alpha } & -1 & -1 & 1 & 0 & 0 & 0 & 0 & 0 & 0 \\
 1+5^{\alpha }+6^{\alpha } & 2^{\alpha }-3^{\alpha } & -2^{\alpha } & -1 & -1 & 1 & 0 & 0 & 0 & 0 & 0 \\
 -2^{\alpha }+5^{\alpha }-6^{\alpha }-7^{\alpha } & 1+3^{\alpha } & 2^{\alpha } & 0 & -1 & -1 & 1 & 0 & 0 & 0 & 0 \\
 1-3^{\alpha }-6^{\alpha }+7^{\alpha } & 3^{\alpha } & 1+2^{\alpha } & -2^{\alpha } & 0 & -1 & -1 & 1 & 0 & 0 & 0 \\
 -2^{\alpha }+7^{\alpha } & 1-2^{\alpha } & -3^{\alpha } & 1+2^{\alpha } & 0 & 0 & -1 & -1 & 1 & 0 & 0 \\
 -3^{\alpha }-5^{\alpha }+10^{\alpha } & -5^{\alpha } & 1+3^{\alpha } & 2^{\alpha } & 1-2^{\alpha } & 0 & 0 & -1 & -1 & 1 & 0 \\
 6^{\alpha }-10^{\alpha }-11^{\alpha } & -2^{\alpha }-3^{\alpha }+5^{\alpha } & -2^{\alpha }+3^{\alpha } & 1 & 2^{\alpha } & 1 & 0 & 0 & -1 & -1 & 1 \\
 \hline
\end{array}
}
\end{equation*}
\end{center} 
\subcaption*{(i) $s_{n,k}$} 
\end{minipage} 

\begin{minipage}{\linewidth} 
\begin{center} 
\tiny 
\begin{equation*} 
\boxed{ 
\begin{array}{cccccccccccccccccc} \hline 
 1 & 0 & 0 & 0 & 0 & 0 & 0 & 0 & 0 & 0 & 0 \\
 1+2^{\alpha } & 1 & 0 & 0 & 0 & 0 & 0 & 0 & 0 & 0 & 0 \\
 1+3^{\alpha } & 1 & 1 & 0 & 0 & 0 & 0 & 0 & 0 & 0 & 0 \\
 2+2^{\alpha }+4^{\alpha } & 1+2^{\alpha } & 1 & 1 & 0 & 0 & 0 & 0 & 0 & 0 & 0 \\
 2+5^{\alpha } & 2 & 1 & 1 & 1 & 0 & 0 & 0 & 0 & 0 & 0 \\
 3+2^{\alpha }+3^{\alpha }+6^{\alpha } & 2+2^{\alpha }+3^{\alpha } & 2+2^{\alpha } & 1 & 1 & 1 & 0 & 0 & 0 & 0 & 0 \\
 4+7^{\alpha } & 3 & 2 & 2 & 1 & 1 & 1 & 0 & 0 & 0 & 0 \\
 5+2^{\alpha +1}+4^{\alpha }+8^{\alpha } & 4+2^{\alpha }+4^{\alpha } & 3+2^{\alpha } & 2+2^{\alpha } & 2 & 1 & 1 & 1 & 0 & 0 & 0 \\
 6+3^{\alpha }+9^{\alpha } & 5+3^{\alpha } & 4+3^{\alpha } & 3 & 2 & 2 & 1 & 1 & 1 & 0 & 0 \\
 8+2^{\alpha +1}+5^{\alpha }+10^{\alpha } & 6+2^{\alpha +1}+5^{\alpha } & 5+2^{\alpha } & 4+2^{\alpha } & 3+2^{\alpha } & 2 & 2 & 1 & 1 & 1 & 0 \\
 10+11^{\alpha } & 8 & 6 & 5 & 4 & 3 & 2 & 2 & 1 & 1 & 1 \\
 \hline 
\end{array}
} 
\end{equation*} 
\end{center} 
\subcaption*{(ii) $s_{n,k}^{(-1)}$} 
\end{minipage}

\caption{The generalized factorization where $(C(q), \gamma(n)) := ((q; q)_{\infty}, n^{\alpha})$} 
\label{figure_genfactpair_spcase_v2} 

\end{figure} 

\begin{example}[Convolutions with the Euler Totient Function]
Suppose that for an arbitrary sequence, $\{a_m\}_{m \geq 1}$, we define the factorization of the 
Lambert series over $a_n$ to be 
\begin{align*} 
\sum_{n \geq 1} \frac{a_n q^n}{1-q^n} & = \frac{1}{(q; q)_{\infty}} 
     \sum_{n \geq 1} \sum_{k=1}^n s_{n,k}(\phi) \widetilde{a}_k(\phi) \cdot q^n, 
\end{align*} 
where we define $s_{n,k}(\phi)$ in terms of its corresponding inverse sequence through 
\eqref{eqn_snk_snkinv_det-based_recrels} given by the divisor sum 
\begin{align*} 
s_{n,k}^{(-1)}(\phi) & := \sum_{d|n} p(d-k) \cdot \phi(n / d). 
\end{align*} 
Then we have an exact formula given in the following form where we note that 
$n = \sum_{d|n} \phi(d)$: 
\begin{align*} 
\widetilde{a}_k(\phi) & = \sum_{d|k} a_d \cdot (k/d) = k \cdot \sum_{d|k} \frac{a_d}{d}. 
\end{align*} 
The two inverse sequences, $s_{n,k}$ and $s_{n,k}^{(-1)}$, in the case of this example are 
listed in Figure \ref{figure_genfactpair_spcase_v1}. 
\end{example} 

\begin{example}[Convolutions with Powers of $n^{\alpha}$]
We extend the generalized expansion in the previous example by choosing our convolution 
function defining $s_{n,k}^{(-1)}$ to be $n^{\alpha}$. More precisely, 
suppose that for an arbitrary sequence, $\{a_m\}_{m \geq 1}$, we again 
define the factorization of the Lambert series over $a_n$ to be 
\begin{align*} 
\sum_{n \geq 1} \frac{a_n q^n}{1-q^n} & = \frac{1}{(q; q)_{\infty}} 
     \sum_{n \geq 1} \sum_{k=1}^n s_{n,k} \widetilde{a}_k \cdot q^n, 
\end{align*} 
where we similarly define $s_{n,k}$ in terms of its corresponding inverse sequence through 
\eqref{eqn_snk_snkinv_det-based_recrels} given by the divisor sum 
\begin{align*} 
s_{n,k}^{(-1)} & := \sum_{d|n} p(d-k) \cdot (n / d)^{\alpha}. 
\end{align*} 
Then we can once again prove that we have an exact formula given in the 
following form where we note that the generalized sum-of-divisors function is defined by 
$\sigma_{\alpha} = \sum_{d|n} d^{\alpha}$: 
\begin{align*} 
\widetilde{a}_k & = \sum_{d|k} a_d \cdot \sigma_{\alpha}(n/d). 
\end{align*} 
The two corresponding inverse sequences, $s_{n,k}$ and $s_{n,k}^{(-1)}$, 
in the case of this particular modified example are listed in 
Figure \ref{figure_genfactpair_spcase_v2}. 
\end{example} 

\subsection{More general theorems} 

The next theorem makes precise a generalized form of the factorization theorem variant 
suggested by the last two examples in the previous subsection. 

\begin{theorem}[Generalized Factorization Theorem I] 
\label{theorem_GenLSFactThm_vI} 
Suppose that the sequence $\{a_n\}_{n \geq 1}$ is taken to be arbitrary and that the functions, 
$C(q)$ and $\gamma(n)$, are fixed. Then we have a generalized Lambert series factorization 
theorem expanded in the form of 
\begin{align*} 
\sum_{n \geq 1} \frac{a_n q^n}{1-q^n} & = \frac{1}{C(q)} \sum_{n \geq 1} \sum_{k=1}^n 
     s_{n,k}(\gamma) \widetilde{a}_k \cdot q^n, 
\end{align*} 
where $s_{n,k}(\gamma)$ is defined through its inverse sequence by \eqref{eqn_snk_snkinv_det-based_recrels}
according to the formula 
\begin{align*} 
s_{n,k}^{(-1)}(\gamma) & := \sum_{d|n} [q^{d-k}] \frac{1}{C(q)} \cdot \gamma(n/d), 
\end{align*} 
and where for $\widetilde{\gamma}(n) := \sum_{d|n} \gamma(d)$ we have that 
\begin{align*} 
\widetilde{a}_k(\gamma) & = \sum_{d|k} a_d \cdot \widetilde{\gamma}(n/d). 
\end{align*} 
\end{theorem} 
\begin{proof} 
Let 
By the same argument justifying the matrix equation in \eqref{eqn_fn_matrix_eqn} from the 
factorization in \eqref{eqn_FundLSFactorizationForm}, we see that 
\begin{align*} 
\widetilde{a}_n & = \sum_{k=1}^n s_{n,k}^{(-1)} \times [q^k]\left( 
     \sum_{d=1}^k \frac{a_d q^d}{1-q^d} C(q)\right). 
\end{align*} 
Thus for fixed $n \geq 1$ and each $1 \leq d \leq n$ we have that 
\begin{align*} 
[a_d] \widetilde{a}_n & = \sum_{k=1}^n s_{n,k}^{(-1)} \times 
     \underset{:= t_{k,d}}{\underbrace{[q^k] \frac{q^d}{1-q^d} C(q)}} \\ 
     & = 
     \sum_{k=d}^n \left(\sum_{r|n} p(r-k) \gamma(n/r)\right) \cdot t_{k,d} \\ 
     & = 
     \sum_{r|n} \left[p(r-d) t_{d,d}+p(r-d-1) t_{d+1,d} + \cdots + p(0) t_{r,d}\right] 
     \gamma(n / r) \\ 
\tag{i} 
     & = 
     \sum_{r|n} \left(\sum_{i=d}^r p(r-i) \cdot t_{i,d}\right) \gamma(n/r). 
\end{align*} 
If we can show that the inner sum is one when $d|r$ where $d|n$ and zero otherwise, we have completed the 
proof of our result. 
We note that for $d \geq 1$ and $i \geq d-1$ we have that $t_{i,d} = [q^{i-d}] C(q)$. 
Then we continue expanding the inner sum in (i) as\footnote{ 
     \underline{\emph{Notation}}: 
     \emph{Iverson's convention} compactly specifies 
     boolean-valued conditions and is equivalent to the 
     \emph{Kronecker delta function}, $\delta_{i,j}$, as 
     $\Iverson{n = k} \equiv \delta_{n,k}$. 
     Similarly, $\Iverson{\mathtt{cond = True}} \equiv 
                 \delta_{\mathtt{cond}, \mathtt{True}}$ 
     in the remainder of the article. 
} 
\begin{align*} 
\tag{ii} 
\sum_{i=d}^r p(r-t) \cdot t_{i,d} & = \sum_{i=0}^{r-d} p(r-d-i) \cdot t_{i+d,d} + 
     \sum_{i=0}^{d-1} p(r-i) \cdot [q^{i-d}] \frac{C(q)}{1-q^d} \\ 
     & = 
     [q^{r-d}] \frac{1}{\cancel{C(q)}} \frac{q^d}{1-q^d} \cancel{C(q)} \\ 
     & = 
     \Iverson{d|r\text{ where }r|n}. 
\end{align*} 
Hence, we have from (i) and (ii) that for $1 \leq d \leq n$ 
\begin{align*} 
[a_d] \widetilde{a_n} & = 
     \begin{cases} 
     \widetilde{\gamma}(n/d) = \sum_{r|\frac{n}{d}} \gamma\left(\frac{n}{dr}\right), & 
     \text{ if $d|n$; } \\ 
     0, & \text{ if $d \not{\mid} n$, } 
     \end{cases} 
\end{align*} 
which implies our formula for $\widetilde{a}_n$ stated in the theorem. 
Here, we notice that it is apparent from the factorization given in the first 
equation of the theorem that $\widetilde{a}_n = \sum_{d=1}^n \gamma_{n,d} \cdot a_d$ 
for some coefficients, $\gamma_{n,d}$, which we have just proved a formula for 
in the previous equation. 
\end{proof} 

\begin{example}[New Convolution Identities from the Matrix Factorizations]
The corresponding matrix factorization representation from \eqref{eqn_fn_matrix_eqn} 
resulting from the theorem provides that for all $n \geq 1$ and fixed 
factorization pair parameter $C(q)$ we have that 
\begin{align} 
\label{eqn_Tildean_MatrixEqn_Ident_examples}
\widetilde{a}_n & = \sum_{k=1}^n s_{n,k}^{(-1)} \cdot B_{k-1}, 
\end{align} 
where $\widetilde{a}_n$, $s_{n,k}^{(-1)}$, and $B_m$ are respectively defined as in 
Theorem \ref{theorem_GenLSFactThm_vI} and Theorem \ref{theorem_MainThm_LSFactProps}. 
One corollary of this result (among many) provides an exact expression for the 
coefficients of the Lambert series over the generalized sum-of-divisors function, 
$\sigma_{\alpha}(n)$, for any fixed $\alpha \in \mathbb{C}$: 
\begin{align*} 
[q^n] & \left(\sum_{m \geq 1} \frac{\sigma_{\alpha}(m) q^m}{1-q^m}\right) = 
     \sum_{d|n} \sigma_{\alpha}(d) \\ 
     & = 
     \sum_{k=1}^n \left(\sum_{d|n} p(d-k) (n/d)^{\alpha}\right) \left[ 
     \sigma_0(k) + \sum_{s = \pm 1} \sum_{j=1}^{\left\lfloor \frac{\sqrt{24k+1}-s}{6} \right\rfloor} 
     \sigma_0\left(k-\frac{j(3j+s)}{2}\right)\right].
\end{align*} 
Similarly, by setting $a_n := n^{\beta}$ and $\gamma(n) := n^{\alpha}$ 
for some fixed $\alpha, \beta \in \mathbb{C}$, we obtain the identity 
\begin{align*} 
     \sum_{d|n} & d^{\beta} \cdot \sigma_{\alpha}(n/d) \\ 
     & = 
     \sum_{k=1}^n \left(\sum_{d|n} p(d-k) (n/d)^{\alpha}\right) \left[ 
     \sigma_{\beta}(k) + \sum_{s = \pm 1} \sum_{j=1}^{\left\lfloor \frac{\sqrt{24k+1}-s}{6} \right\rfloor} 
     \sigma_{\beta}\left(k-\frac{j(3j+s)}{2}\right)\right].
\end{align*} 
If we set $(a_n, \gamma(n)) := (n^{\beta}, \phi(n))$ for a fixed $\beta$, we obtain the following 
related identity: 
\begin{align*} 
     & n \cdot \sigma_{\beta-1}(n) \\ 
     & = 
     \sum_{k=1}^n \left(\sum_{d|n} p(d-k) \phi(n/d)\right) \left[ 
     \sigma_{\beta}(k) + \sum_{s = \pm 1} \sum_{j=1}^{\left\lfloor \frac{\sqrt{24k+1}-s}{6} \right\rfloor} 
     \sigma_{\beta}\left(k-\frac{j(3j+s)}{2}\right)\right].
\end{align*} 
Given the known special case Lambert series identities expanded in \cite[\S 1; \cf \S 3]{MERCA-SCHMIDT1}, 
we also have the following mixed bag of additional identities which are easily proved from 
\eqref{eqn_Tildean_MatrixEqn_Ident_examples} for fixed $\beta, t \in \mathbb{C}$: 
\begin{align*} 
 & \sum_{d|n} \log(d) \\ 
 & \quad = \sum_{k=1}^n \left(\sum_{d|n} p(d-k) \Lambda(n/d)\right) \left[ 
     \sigma_{0}(k) + \sum_{s = \pm 1} \sum_{j=1}^{\left\lfloor \frac{\sqrt{24k+1}-s}{6} \right\rfloor} 
     \sigma_{0}\left(k-\frac{j(3j+s)}{2}\right)\right] \\ 
 & \log(n) \sigma_{\beta}(n) - \sum_{d|n} d^{\beta} \log(d) \\ 
 & \quad = \sum_{k=1}^n \left(\sum_{d|n} p(d-k) \Lambda(n/d)\right) \left[ 
     \sigma_{\beta}(k) + \sum_{s = \pm 1} \sum_{j=1}^{\left\lfloor \frac{\sqrt{24k+1}-s}{6} \right\rfloor} 
     \sigma_{\beta}\left(k-\frac{j(3j+s)}{2}\right)\right] \\ 
 & \Lambda(n) \\ 
 & \quad = \sum_{k=1}^n \left(\sum_{d|n} p(d-k) \Lambda(n/d)\right) \left[ 
     \Iverson{k=1} + \sum_{s = \pm 1} \sum_{j=1}^{\left\lfloor \frac{\sqrt{24k+1}-s}{6} \right\rfloor} 
     \Iverson{k-\frac{j(3j+s)}{2}=1}\right] \\ 
 & n^{t} \cdot \sigma_{\beta-t}(n) \\ 
 & \quad = \sum_{k=1}^n \left(\sum_{d|n} p(d-k) J_t(n/d)\right) \left[ 
     \sigma_{\beta}(k) + \sum_{s = \pm 1} \sum_{j=1}^{\left\lfloor \frac{\sqrt{24k+1}-s}{6} \right\rfloor} 
     \sigma_{\beta}\left(k-\frac{j(3j+s)}{2}\right)\right]. 
\end{align*}
We note that these identities implicitly involving the Euler partition function $p(n)$ correspond to the 
choice of the factorization pair parameter $C(q) := (q; q)_{\infty}$. We could just as easily re-phrase these 
expansions in terms of the partition function $q(n)$ where $C(q) = 1 / (-q; q)_{\infty}$, or in terms of 
any number of other special sequences with a reciprocal generating function of $C(q)$. 
\end{example} 

\subsection{A second variation of the theorem}

\begin{table}[ht!] 
\centering

\small 
\begin{equation*} 
\begin{array}{|c||ll|} \hline\hline
m & \widetilde{a}_m^{\prime} & \sum_{d|m} \widetilde{a}_d^{\prime} \\ \hline\hline 
 1 & a_1 & a_1 \\
 2 & a_2-4 a_1 & a_2-3 a_1 \\
 3 & -2 a_1-2 a_2+a_3 & -a_1-2 a_2+a_3 \\
 4 & 9 a_1-3 a_2-2 a_3+a_4 & 6 a_1-2 a_2-2 a_3+a_4 \\
 5 & -2 a_1+4 a_2-a_3-2 a_4+a_5 & -a_1+4 a_2-a_3-2 a_4+a_5 \\
 6 & 13 a_1+a_2-a_4-2 a_5+a_6 & 8 a_1+a_3-a_4-2 a_5+a_6 \\
 7 & -15 a_1+4 a_2+3 a_3+2 a_4-a_5-2 a_6+a_7 & -14 a_1+4 a_2+3 a_3+2 a_4-a_5-2 a_6+a_7 \\
 8 & -8 a_1+5 a_3-a_4+2 a_5-a_6-2 a_7+a_8 & -2 a_1-2 a_2+3 a_3+2 a_5-a_6-2 a_7+a_8 \\
\hline\hline 
\end{array}
\end{equation*} 

\caption{Example of a second variant of the generalized factorization theorem where 
         $(C(q), \gamma(n)) := ((q; q)_{\infty}, \phi(n))$.} 
\label{table_example_GenFactThmII_spcase_PnCVLEulerPhi} 

\end{table} 

\begin{example}[A Second Variation of the Theorem] 
Let the sequence $\{a_n\}_{n \geq 1}$ be fixed and suppose that the functions, 
$C(q) := (q; q)_{\infty}$ and $\gamma(n) := \phi(n)$. 
Then we have another construction of a  generalized Lambert series factorization 
theorem for these parameters expanded in the form of 
\begin{align*} 
\sum_{n \geq 1} \frac{\widetilde{a}^{\prime}_n q^n}{1-q^n} & = \frac{1}{(q; q)_{\infty}} \sum_{n \geq 1} \sum_{k=1}^n 
     s_{n,k}(\phi) a_k \cdot q^n, 
\end{align*} 
where $s_{n,k}(\phi)$ is defined through its inverse sequence by 
\eqref{eqn_snk_snkinv_det-based_recrels}according to the formula 
\begin{align*} 
s_{n,k}^{(-1)}(\phi) & := \sum_{d|n} p(d-k) \cdot \phi(n/d). 
\end{align*} 
Table \ref{table_example_GenFactThmII_spcase_PnCVLEulerPhi} provides the first several cases of the 
right-hand-side terms, $\widetilde{a}^{\prime}_n$, and the divisor sums over this sequence that define the 
series coefficients of the right-hand-side Lambert series expansion in the second to last equation. 
\end{example} 

We generalize the results in the preceding example by the next theorem. 

\begin{theorem}[Generalized Factorization Theorem II] 
\label{theorem_GenLSFactThm_vII}
Suppose that the sequence $\{a_n\}_{n \geq 1}$ is taken to be arbitrary and that the functions, 
$C(q)$ and $\gamma(n)$, are fixed. Then we have a generalized Lambert series factorization 
theorem expanded in the form of 
\begin{align*} 
\sum_{n \geq 1} \frac{\widetilde{a}^{\prime}_n q^n}{1-q^n} & = \frac{1}{C(q)} \sum_{n \geq 1} \sum_{k=1}^n 
     s_{n,k}(\gamma) a_k \cdot q^n, 
\end{align*} 
where $s_{n,k}(\gamma)$ is defined through its inverse sequence by 
\eqref{eqn_snk_snkinv_det-based_recrels} according to the formula 
\begin{align*} 
s_{n,k}^{(-1)}(\gamma) & := \sum_{d|n} [q^{d-k}] \frac{1}{C(q)} \cdot \gamma(n/d), 
\end{align*} 
and where we have that for all $m \geq 1$
\begin{align*} 
\sum_{d|m} \widetilde{a}_d^{\prime}(\gamma) & = \sum_{i=1}^m \sum_{j=1}^{m+1-i} 
     a_i \cdot s_{m+1-j,i} \cdot [q^{j-1}] \frac{1}{C(q)}. 
\end{align*} 
\end{theorem} 
\begin{proof} 
We equate the left-hand-side to the right-hand-side of the theorem statement to obtain the 
expansions 
\begin{align*} 
\sum_{d|n} \widetilde{a}^{\prime}_d & = [q^n]\left(\sum_{n \geq 1} \frac{\widetilde{a}^{\prime}_n q^n}{1-q^n}\right) \\ 
     & = 
     \sum_{j=0}^n \sum_{k=1}^{n-j} s_{n-j,k} a_k \cdot [q^j] \frac{1}{C(q)} \\ 
     & = 
     \sum_{k=1}^n \sum_{j=0}^n s_{n-j,k} a_k \cdot [q^j] \frac{1}{C(q)} \\ 
     & = 
     \sum_{k=1}^n \sum_{j=0}^{n-k} s_{n-j,k} a_k \cdot [q^j] \frac{1}{C(q)}, 
\end{align*} 
since $s_{n-j,k}$ is zero-valued for $n-j < k$ which requires that for $s_{n,k}$ to be 
potentially non-zero we must have that $n-j \geq k$, or equivalently that 
$n-k \geq j$ as the upper bound of the inner sum with respect to $j$. 
Shifting the index of summation in the inner sum by one then leads to the identity for these 
Lambert series coefficients over powers of $q^n$. Hence we have proved the theorem. 
\end{proof}

\section{Conclusions} 
\label{Section_Concl} 

\subsection{Summary} 

In Section \ref{Section_NaturalGensOfFactorPairs} and Section \ref{Section_VariantsOfTheFactThms} 
we proved several new forms of the Lambert 
series factorization theorem in \eqref{eqn_FundLSFactorizationForm} 
which is defined by the dependent factor pair 
parameters, $C(q)$ and $s_{n,k}$. The interpretation of these theorems provides a 
corresponding matrix factorization which effectively generalizes the known result in 
\eqref{eqn_fn_matrix_eqn} from \cite{MERCA-SCHMIDT1,SCHMIDT-LSFACTTHM}. 
The first theorems proved in Section \ref{Section_NaturalGensOfFactorPairs} also lead 
to a number of new summation identities connecting partition functions such as $p(n)$ 
with sums over special multiplicative functions with well-known Lambert series expansions 
found in the literature \cite[\cf \S 1, \S 3]{MERCA-SCHMIDT1}. 
The generalizations of the first pair of theorems we proved later in the 
variations of Section \ref{Section_VariantsOfTheFactThms} 
provide yet additional interpretations and identities between sums of the 
functions implicit to \eqref{eqn_LambertSeriesfb_def}, generalized partition functions, and 
other special multiplicative functions of importance in number theory. 

\subsection{Even more general factorization theorems} 

\begin{table}[ht!] 
\centering

\small 
\begin{equation*} 
\begin{array}{|cccccccc|} \hline\hline
 0 & 0 & 0 & 0 & 0 & 0 & 0 & 0 \\
 0 & 0 & 0 & 0 & 0 & 0 & 0 & 0 \\
 1 & 0 & 0 & 0 & 0 & 0 & 0 & 0 \\
 -1 & 0 & 0 & 0 & 0 & 0 & 0 & 0 \\
 -1 & 1 & 0 & 0 & 0 & 0 & 0 & 0 \\
 d & -1 & 0 & 0 & 0 & 0 & 0 & 0 \\
 -d & -1 & 1 & 0 & 0 & 0 & 0 & 0 \\
 1-d & 0 & -1 & 0 & 0 & 0 & 0 & 0 \\
 d^2 & 0 & -1 & 1 & 0 & 0 & 0 & 0 \\
 1-d^2 & d+1 & 0 & -1 & 0 & 0 & 0 & 0 \\
 d-d^2 & -d & 0 & -1 & 1 & 0 & 0 & 0 \\
 d^3 & 1-d & 1 & 0 & -1 & 0 & 0 & 0 \\
 d-d^3 & 0 & 0 & 0 & -1 & 1 & 0 & 0 \\
 d^2-d^3 & 0 & d+1 & 1 & 0 & -1 & 0 & 0 \\
 d^4-1 & d (d+1) & -d & 0 & 0 & -1 & 1 & 0 \\
 d^2-d^4 & -d^2 & -d & 1 & 1 & 0 & -1 & 0 \\
\hline\hline 
\end{array}
\end{equation*} 

\caption{Example of the factorization sequence $s_{n,k}(d)$ for the 
         factorization of $L(c, d; 2, 1, 2, 1)$ when $C(q) := (q; q)_{\infty}$ 
         defined on page \pageref{eqn_Lcd2121_factorization_thm}.} 
\label{table_GenLSeries_alphagammaEQ21} 

\end{table} 

\subsubsection{Expansions of generalized Lambert series} 

We seek to generalize the factorization theorem result in 
\eqref{eqn_FundLSFactorizationForm} to a corresponding form 
for the following generalized Lambert series expansions for fixed constants 
$c, d, \alpha, \beta, \gamma, \delta \in \mathbb{C}$ defined such that the series converges: 
\begin{align} 
\label{eqn_GenLambertSeries_LacdAlphaBetaGammaDelta_exp_v1} 
L_a(c, d; \alpha, \beta, \gamma, \delta) & := 
     \sum_{n \geq 1} \frac{a_n c^n q^{\alpha n+\beta}}{1-d \cdot q^{\gamma n+\delta}}. 
\end{align} 
It is not difficult to show that the series coefficients of $q^n$ in the previous Lambert series 
expansion are given in closed-form according to the special case formula 
\[
[q^n] L_a(c, d; \alpha, \gamma, \alpha, \gamma) = 
     \sum_{\substack{\alpha m+\gamma \\ m \geq \alpha+\gamma}} c^m a_m \cdot d^{\frac{n}{\alpha m+\gamma}-1}. 
\]
Applications of a corresponding factorization result include new identities for the generalized 
Lambert series generating the sum-of-squares function, $r_2(n)$, in the form of 
\cite[\S 17.10]{HARDYANDWRIGHT} 
\begin{align*} 
\sum_{m \geq 1} r_2(m) q^m & = \sum_{n \geq 1} \frac{4 \cdot (-1)^{n+1} q^{2n+1}}{1-q^{2n+1}}. 
\end{align*} 
For example, we may formulate a generalized variant of the factorization theorems in this article as 
\begin{align*} 
\sum_{n \geq 1} \frac{a_n c^n q^{2n+1}}{1-d \cdot q^{2n+1}} & = 
     \frac{1}{C(q)} \sum_{n \geq 1} \sum_{k=1}^n s_{n,k}(d) c^k a_k \cdot q^n, 
\end{align*} 
where for an arbitrary sequence, $\{a_n\}_{n \geq 1}$, we have that the series coefficients of the 
left-hand-side Lambert series in the previous equation are given by 
\label{eqn_Lcd2121_factorization_thm}
\begin{align*}
[q^n] L(c, d; 2, 1, 2, 1) & = \sum_{\substack{2m+1|n \\ m>1}} c^m a_m \cdot d^{\frac{n}{2m+1}-1}. 
\end{align*} 
For $C(q) := (q; q)_{\infty}$, an example of the factorization in the second to last equation is 
shown in Table \ref{table_GenLSeries_alphagammaEQ21}. 

The expansions of the generalized Lambert series in 
\eqref{eqn_GenLambertSeries_LacdAlphaBetaGammaDelta_exp_v1} also allow us to approach new identities for the 
Lambert series generating the logarithmic derivatives of the Jacobi theta functions in the forms of 
\cite[\S 20.5(ii)]{NISTHB} 
\begin{align*} 
\frac{\vartheta_1^{\prime}(z, q)}{\vartheta_1(z, q)} & = 
     4 \sum_{n \geq 1} \frac{\sin(2nz) q^{2n}}{1-q^{2n}} + \cot(z) \\ 
\frac{\vartheta_2^{\prime}(z, q)}{\vartheta_2(z, q)} & = 
     4 \sum_{n \geq 1} \frac{(-1)^n \sin(2nz) q^{2n}}{1-q^{2n}} -\tan(z) \\ 
\frac{\vartheta_3^{\prime}(z, q)}{\vartheta_3(z, q)} & = 
     4 \sum_{n \geq 1} \frac{(-1)^n \sin(2nz) q^{n}}{1-q^{2n}} \\ 
\frac{\vartheta_4^{\prime}(z, q)}{\vartheta_4(z, q)} & = 
     4 \sum_{n \geq 1} \frac{\sin(2nz) q^{n}}{1-q^{2n}}. 
\end{align*} 
Similarly, by considering derivatives of the generalized Lambert series as in 
\cite{SCHMIDT-COMBSUMSBDDDIV}, we can generate higher-order cases of the 
derivatives of the Jacobi theta functions, including the following identities 
\cite[\S 20.4(ii)]{NISTHB}: 
\begin{align*} 
\frac{\vartheta_1^{\prime\prime\prime}(0, q)}{\vartheta_1^{\prime}(0, q)} & = 
     -1 + 24 \sum_{n \geq 1} \frac{q^{2n}}{(1-q^{2n})^2} \\ 
\frac{\vartheta_2^{\prime\prime}(0, q)}{\vartheta_2(0, q)} & = 
     -1 - 8 \sum_{n \geq 1} \frac{q^{2n}}{(1+q^{2n})^2} \\ 
\frac{\vartheta_3^{\prime\prime}(0, q)}{\vartheta_3(0, q)} & = 
     - 8 \sum_{n \geq 1} \frac{q^{2n-1}}{(1+q^{2n-1})^2} \\ 
\frac{\vartheta_4^{\prime\prime}(0, q)}{\vartheta_4(0, q)} & = 
     8 \sum_{n \geq 1} \frac{q^{2n-1}}{(1-q^{2n-1})^2}. 
\end{align*} 

\subsubsection{Transformations of Lambert series} 

One possible transformation providing an application of the generalized factorization theorems we 
have already proved within this article is given by 
\begin{align*}
\sum_{n=1}^{\infty} \frac{a_n q^n}{1+q^n} 
&  = \sum_{n=1}^{\infty} \frac{a_n q^n}{1-q^n}  - 2\sum_{n=1}^{\infty} \frac{a_n q^{2n}}{1-q^{2n}} \\
&  = \sum_{n=1}^{\infty} \frac{b_n q^n}{1-q^n} ,
\end{align*}
where 
$$
b_n = \begin{cases}
a_n, & \text{for $n$ odd,}\\
a_n-2a_{n/2} & \text{for $n$ even.}\\
\end{cases}
$$
We can similarly generate the terms in the slightly more general Lambert series expansions of 
\begin{align*} 
\sum_{n=1}^{\infty} \frac{a_n \cdot c^n q^n}{1 \pm q^n},\ \max(|cq|, |q|) < 1, 
\end{align*} 
by making the substitution of $a_k \mapsto c^k a_k$ in the factorization theorems proved above. 

\subsubsection{Topics for future research and investigation} 

The generalizations to the factorization theorems we have proved in this article suggested in this 
section comprise a new avenue of future research based on our new results. 
We anticipate that the investigation of these topics will be a 
fruitful source of new identities and insights to other special multiplicative functions enumerated by 
Lambert series generating functions of the forms defined by 
\eqref{eqn_GenLambertSeries_LacdAlphaBetaGammaDelta_exp_v1}. 

\subsection*{Acknowledgments} 

The authors thank the referees for their helpful insights and comments on 
preparing the manuscript.

\newpage 
\renewcommand{\thesection}{A}
\section{Appendix: Additional tables and figures} 
\label{AppendixA_Tables_Figures} 

\begin{figure}[ht!]

\begin{minipage}{\linewidth} 
\begin{center} 
\tiny
\begin{equation*} 
\boxed{ 
\begin{array}{cccccccccccccccccc} \hline 
 1 & 0 & 0 & 0 & 0 & 0 & 0 & 0 & 0 & 0 & 0 & 0 & 0 & 0 & 0 & 0 & 0 & 0 \\
 0 & 1 & 0 & 0 & 0 & 0 & 0 & 0 & 0 & 0 & 0 & 0 & 0 & 0 & 0 & 0 & 0 & 0 \\
 0 & -1 & 1 & 0 & 0 & 0 & 0 & 0 & 0 & 0 & 0 & 0 & 0 & 0 & 0 & 0 & 0 & 0 \\
 -1 & 1 & -1 & 1 & 0 & 0 & 0 & 0 & 0 & 0 & 0 & 0 & 0 & 0 & 0 & 0 & 0 & 0 \\
 0 & -2 & 0 & -1 & 1 & 0 & 0 & 0 & 0 & 0 & 0 & 0 & 0 & 0 & 0 & 0 & 0 & 0 \\
 -1 & 2 & 0 & 0 & -1 & 1 & 0 & 0 & 0 & 0 & 0 & 0 & 0 & 0 & 0 & 0 & 0 & 0 \\
 0 & -3 & 0 & -1 & 0 & -1 & 1 & 0 & 0 & 0 & 0 & 0 & 0 & 0 & 0 & 0 & 0 & 0 \\
 -1 & 3 & -1 & 2 & -1 & 0 & -1 & 1 & 0 & 0 & 0 & 0 & 0 & 0 & 0 & 0 & 0 & 0 \\
 1 & -4 & 1 & -2 & 1 & -1 & 0 & -1 & 1 & 0 & 0 & 0 & 0 & 0 & 0 & 0 & 0 & 0 \\
 -1 & 5 & -1 & 1 & 0 & 1 & -1 & 0 & -1 & 1 & 0 & 0 & 0 & 0 & 0 & 0 & 0 & 0 \\
 1 & -6 & 1 & -2 & 0 & -1 & 1 & -1 & 0 & -1 & 1 & 0 & 0 & 0 & 0 & 0 & 0 & 0 \\
 -1 & 7 & -1 & 4 & -1 & 2 & -1 & 1 & -1 & 0 & -1 & 1 & 0 & 0 & 0 & 0 & 0 & 0 \\
 2 & -8 & 1 & -4 & 1 & -2 & 1 & -1 & 1 & -1 & 0 & -1 & 1 & 0 & 0 & 0 & 0 & 0 \\
 -1 & 10 & -1 & 3 & -1 & 2 & 0 & 1 & -1 & 1 & -1 & 0 & -1 & 1 & 0 & 0 & 0 & 0 \\
 2 & -11 & 2 & -4 & 2 & -3 & 1 & -1 & 1 & -1 & 1 & -1 & 0 & -1 & 1 & 0 & 0 & 0 \\
 -2 & 13 & -2 & 7 & -2 & 3 & -2 & 3 & -1 & 1 & -1 & 1 & -1 & 0 & -1 & 1 & 0 & 0 \\
 3 & -15 & 2 & -7 & 2 & -3 & 1 & -3 & 2 & -1 & 1 & -1 & 1 & -1 & 0 & -1 & 1 & 0 \\
 -2 & 18 & -2 & 6 & -2 & 5 & -1 & 2 & -1 & 2 & -1 & 1 & -1 & 1 & -1 & 0 & -1 & 1 \\
 \hline
\end{array}
}
\end{equation*}
\end{center} 
\subcaption*{(i) $s_{n,k}$} 
\end{minipage} 

\begin{minipage}{\linewidth} 
\begin{center} 
\tiny 
\begin{equation*} 
\boxed{ 
\begin{array}{cccccccccccccccccc} \hline 
 1 & 0 & 0 & 0 & 0 & 0 & 0 & 0 & 0 & 0 & 0 & 0 & 0 & 0 & 0 & 0 & 0 & 0 \\
 0 & 1 & 0 & 0 & 0 & 0 & 0 & 0 & 0 & 0 & 0 & 0 & 0 & 0 & 0 & 0 & 0 & 0 \\
 0 & 1 & 1 & 0 & 0 & 0 & 0 & 0 & 0 & 0 & 0 & 0 & 0 & 0 & 0 & 0 & 0 & 0 \\
 1 & 0 & 1 & 1 & 0 & 0 & 0 & 0 & 0 & 0 & 0 & 0 & 0 & 0 & 0 & 0 & 0 & 0 \\
 1 & 2 & 1 & 1 & 1 & 0 & 0 & 0 & 0 & 0 & 0 & 0 & 0 & 0 & 0 & 0 & 0 & 0 \\
 2 & 0 & 1 & 1 & 1 & 1 & 0 & 0 & 0 & 0 & 0 & 0 & 0 & 0 & 0 & 0 & 0 & 0 \\
 3 & 3 & 2 & 2 & 1 & 1 & 1 & 0 & 0 & 0 & 0 & 0 & 0 & 0 & 0 & 0 & 0 & 0 \\
 3 & 3 & 2 & 1 & 2 & 1 & 1 & 1 & 0 & 0 & 0 & 0 & 0 & 0 & 0 & 0 & 0 & 0 \\
 5 & 4 & 3 & 3 & 2 & 2 & 1 & 1 & 1 & 0 & 0 & 0 & 0 & 0 & 0 & 0 & 0 & 0 \\
 6 & 3 & 4 & 3 & 2 & 2 & 2 & 1 & 1 & 1 & 0 & 0 & 0 & 0 & 0 & 0 & 0 & 0 \\
 9 & 8 & 6 & 5 & 4 & 3 & 2 & 2 & 1 & 1 & 1 & 0 & 0 & 0 & 0 & 0 & 0 & 0 \\
 8 & 8 & 5 & 4 & 4 & 3 & 3 & 2 & 2 & 1 & 1 & 1 & 0 & 0 & 0 & 0 & 0 & 0 \\
 14 & 12 & 10 & 8 & 6 & 5 & 4 & 3 & 2 & 2 & 1 & 1 & 1 & 0 & 0 & 0 & 0 & 0 \\
 14 & 11 & 10 & 8 & 7 & 5 & 4 & 4 & 3 & 2 & 2 & 1 & 1 & 1 & 0 & 0 & 0 & 0 \\
 20 & 15 & 13 & 11 & 9 & 8 & 6 & 5 & 4 & 3 & 2 & 2 & 1 & 1 & 1 & 0 & 0 & 0 \\
 22 & 18 & 15 & 13 & 10 & 9 & 7 & 5 & 5 & 4 & 3 & 2 & 2 & 1 & 1 & 1 & 0 & 0 \\
 31 & 27 & 22 & 18 & 15 & 12 & 10 & 8 & 6 & 5 & 4 & 3 & 2 & 2 & 1 & 1 & 1 & 0 \\
 30 & 26 & 22 & 18 & 15 & 12 & 11 & 9 & 7 & 6 & 5 & 4 & 3 & 2 & 2 & 1 & 1 & 1 \\
 \hline 
\end{array}
} 
\end{equation*} 
\end{center} 
\subcaption*{(ii) $s_{n,k}^{(-1)}$} 
\end{minipage}

\caption{The factorization where $C(q) := (-q; q)_{\infty}^{-1}$} 
\label{figure_factpair_v1} 

\end{figure} 

\begin{figure}[ht!]

\begin{minipage}{\linewidth} 
\begin{center} 
\small
\begin{equation*} 
\boxed{ 
\begin{array}{llllll} \hline 
 \frac{1}{1-a} & 0 & 0 & 0 & 0 & 0 \\
 \frac{-a-1}{a-1} & \frac{1}{1-a} & 0 & 0 & 0 & 0 \\
 \frac{-a^2-2 a-1}{a-1} & -\frac{a}{a-1} & \frac{1}{1-a} & 0 & 0 & 0 \\
 \frac{-a^3-2 a^2-3 a-1}{a-1} & \frac{-a^2-a-1}{a-1} & -\frac{a}{a-1} & \frac{1}{1-a} & 0 & 0 \\
 \frac{-a^4-2 a^3-4 a^2-4 a-1}{a-1} & -\frac{a \left(a^2+a+2\right)}{a-1} & \frac{-a^2-a}{a-1} & -\frac{a}{a-1} & \frac{1}{1-a} & 0 \\
 \frac{-a^5-2 a^4-4 a^3-6 a^2-5 a-1}{a-1} & \frac{-a^4-a^3-3 a^2-2 a-1}{a-1} & \frac{-a^3-a^2-a-1}{a-1} & \frac{-a^2-a}{a-1} & -\frac{a}{a-1} & \frac{1}{1-a} \\
 \hline
\end{array}
}
\end{equation*}
\end{center} 
\subcaption*{(i) $s_{n,k}$} 
\end{minipage} 

\begin{minipage}{\linewidth} 
\begin{center} 
\small 
\begin{equation*} 
\boxed{ 
\begin{array}{llllll} \hline 
 1-a & 0 & 0 & 0 & 0 & 0 \\
 a^2-1 & 1-a & 0 & 0 & 0 & 0 \\
 a^2-1 & a^2-a & 1-a & 0 & 0 & 0 \\
 a^2-a^3 & a^2-1 & a^2-a & 1-a & 0 & 0 \\
 -a^3+2 a^2-1 & -a^3+2 a^2-a & a^2-a & a^2-a & 1-a & 0 \\
 -2 a^3+a^2+1 & -a^3+a^2+a-1 & -a^3+2 a^2-1 & a^2-a & a^2-a & 1-a \\
 \hline 
\end{array}
} 
\end{equation*} 
\end{center} 
\subcaption*{(ii) $s_{n,k}^{(-1)}$} 
\end{minipage}

\caption{The factorization where $C(q) := (a; q)_{\infty}^{-1}$} 
\label{figure_factpair_v1} 

\end{figure} 

\begin{figure}[ht!]

\begin{minipage}{\linewidth} 
\begin{center} 
\tiny
\begin{equation*} 
\boxed{ 
\begin{array}{cccccccccccccccccc} \hline 
 1 & 0 & 0 & 0 & 0 & 0 & 0 & 0 & 0 & 0 & 0 & 0 & 0 & 0 & 0 & 0 \\
 0 & 1 & 0 & 0 & 0 & 0 & 0 & 0 & 0 & 0 & 0 & 0 & 0 & 0 & 0 & 0 \\
 0 & -1 & 1 & 0 & 0 & 0 & 0 & 0 & 0 & 0 & 0 & 0 & 0 & 0 & 0 & 0 \\
 -1 & 1 & -1 & 1 & 0 & 0 & 0 & 0 & 0 & 0 & 0 & 0 & 0 & 0 & 0 & 0 \\
 0 & -2 & 0 & -1 & 1 & 0 & 0 & 0 & 0 & 0 & 0 & 0 & 0 & 0 & 0 & 0 \\
 -1 & 2 & 0 & 0 & -1 & 1 & 0 & 0 & 0 & 0 & 0 & 0 & 0 & 0 & 0 & 0 \\
 0 & -3 & 0 & -1 & 0 & -1 & 1 & 0 & 0 & 0 & 0 & 0 & 0 & 0 & 0 & 0 \\
 -1 & 3 & -1 & 2 & -1 & 0 & -1 & 1 & 0 & 0 & 0 & 0 & 0 & 0 & 0 & 0 \\
 1 & -4 & 1 & -2 & 1 & -1 & 0 & -1 & 1 & 0 & 0 & 0 & 0 & 0 & 0 & 0 \\
 -1 & 5 & -1 & 1 & 0 & 1 & -1 & 0 & -1 & 1 & 0 & 0 & 0 & 0 & 0 & 0 \\
 1 & -6 & 1 & -2 & 0 & -1 & 1 & -1 & 0 & -1 & 1 & 0 & 0 & 0 & 0 & 0 \\
 -1 & 7 & -1 & 4 & -1 & 2 & -1 & 1 & -1 & 0 & -1 & 1 & 0 & 0 & 0 & 0 \\
 2 & -8 & 1 & -4 & 1 & -2 & 1 & -1 & 1 & -1 & 0 & -1 & 1 & 0 & 0 & 0 \\
 -1 & 10 & -1 & 3 & -1 & 2 & 0 & 1 & -1 & 1 & -1 & 0 & -1 & 1 & 0 & 0 \\
 2 & -11 & 2 & -4 & 2 & -3 & 1 & -1 & 1 & -1 & 1 & -1 & 0 & -1 & 1 & 0 \\
 -2 & 13 & -2 & 7 & -2 & 3 & -2 & 3 & -1 & 1 & -1 & 1 & -1 & 0 & -1 & 1 \\
 \hline
\end{array}
}
\end{equation*}
\end{center} 
\subcaption*{(i) $s_{n,k}$} 
\end{minipage} 

\begin{minipage}{\linewidth} 
\begin{center} 
\tiny 
\begin{equation*} 
\boxed{ 
\begin{array}{cccccccccccccccccc} \hline 
 1 & 0 & 0 & 0 & 0 & 0 & 0 & 0 & 0 & 0 & 0 & 0 & 0 & 0 & 0 & 0 \\
 0 & 1 & 0 & 0 & 0 & 0 & 0 & 0 & 0 & 0 & 0 & 0 & 0 & 0 & 0 & 0 \\
 0 & 1 & 1 & 0 & 0 & 0 & 0 & 0 & 0 & 0 & 0 & 0 & 0 & 0 & 0 & 0 \\
 1 & 0 & 1 & 1 & 0 & 0 & 0 & 0 & 0 & 0 & 0 & 0 & 0 & 0 & 0 & 0 \\
 1 & 2 & 1 & 1 & 1 & 0 & 0 & 0 & 0 & 0 & 0 & 0 & 0 & 0 & 0 & 0 \\
 2 & 0 & 1 & 1 & 1 & 1 & 0 & 0 & 0 & 0 & 0 & 0 & 0 & 0 & 0 & 0 \\
 3 & 3 & 2 & 2 & 1 & 1 & 1 & 0 & 0 & 0 & 0 & 0 & 0 & 0 & 0 & 0 \\
 3 & 3 & 2 & 1 & 2 & 1 & 1 & 1 & 0 & 0 & 0 & 0 & 0 & 0 & 0 & 0 \\
 5 & 4 & 3 & 3 & 2 & 2 & 1 & 1 & 1 & 0 & 0 & 0 & 0 & 0 & 0 & 0 \\
 6 & 3 & 4 & 3 & 2 & 2 & 2 & 1 & 1 & 1 & 0 & 0 & 0 & 0 & 0 & 0 \\
 9 & 8 & 6 & 5 & 4 & 3 & 2 & 2 & 1 & 1 & 1 & 0 & 0 & 0 & 0 & 0 \\
 8 & 8 & 5 & 4 & 4 & 3 & 3 & 2 & 2 & 1 & 1 & 1 & 0 & 0 & 0 & 0 \\
 14 & 12 & 10 & 8 & 6 & 5 & 4 & 3 & 2 & 2 & 1 & 1 & 1 & 0 & 0 & 0 \\
 14 & 11 & 10 & 8 & 7 & 5 & 4 & 4 & 3 & 2 & 2 & 1 & 1 & 1 & 0 & 0 \\
 20 & 15 & 13 & 11 & 9 & 8 & 6 & 5 & 4 & 3 & 2 & 2 & 1 & 1 & 1 & 0 \\
 22 & 18 & 15 & 13 & 10 & 9 & 7 & 5 & 5 & 4 & 3 & 2 & 2 & 1 & 1 & 1 \\
 \hline 
\end{array}
} 
\end{equation*} 
\end{center} 
\subcaption*{(ii) $s_{n,k}^{(-1)}$} 
\end{minipage}

\caption{The factorization where $C(q) := (q; q^2)_{\infty}$} 
\label{figure_factpair_v3} 

\end{figure} 

\begin{figure}[ht!]

\begin{minipage}{\linewidth} 
\begin{center} 
\tiny
\begin{equation*} 
\boxed{ 
\begin{array}{cccccccccccccccccc} \hline 
 1 & 0 & 0 & 0 & 0 & 0 & 0 & 0 & 0 & 0 & 0 & 0 & 0 & 0 & 0 & 0 \\
 2 & 1 & 0 & 0 & 0 & 0 & 0 & 0 & 0 & 0 & 0 & 0 & 0 & 0 & 0 & 0 \\
 3 & 1 & 1 & 0 & 0 & 0 & 0 & 0 & 0 & 0 & 0 & 0 & 0 & 0 & 0 & 0 \\
 5 & 2 & 1 & 1 & 0 & 0 & 0 & 0 & 0 & 0 & 0 & 0 & 0 & 0 & 0 & 0 \\
 7 & 3 & 1 & 1 & 1 & 0 & 0 & 0 & 0 & 0 & 0 & 0 & 0 & 0 & 0 & 0 \\
 10 & 4 & 3 & 1 & 1 & 1 & 0 & 0 & 0 & 0 & 0 & 0 & 0 & 0 & 0 & 0 \\
 14 & 6 & 3 & 2 & 1 & 1 & 1 & 0 & 0 & 0 & 0 & 0 & 0 & 0 & 0 & 0 \\
 19 & 8 & 4 & 3 & 2 & 1 & 1 & 1 & 0 & 0 & 0 & 0 & 0 & 0 & 0 & 0 \\
 25 & 11 & 7 & 4 & 2 & 2 & 1 & 1 & 1 & 0 & 0 & 0 & 0 & 0 & 0 & 0 \\
 33 & 14 & 8 & 5 & 4 & 2 & 2 & 1 & 1 & 1 & 0 & 0 & 0 & 0 & 0 & 0 \\
 43 & 19 & 10 & 7 & 5 & 3 & 2 & 2 & 1 & 1 & 1 & 0 & 0 & 0 & 0 & 0 \\
 55 & 24 & 15 & 9 & 6 & 5 & 3 & 2 & 2 & 1 & 1 & 1 & 0 & 0 & 0 & 0 \\
 70 & 31 & 18 & 12 & 8 & 6 & 4 & 3 & 2 & 2 & 1 & 1 & 1 & 0 & 0 & 0 \\
 88 & 39 & 22 & 15 & 10 & 7 & 6 & 4 & 3 & 2 & 2 & 1 & 1 & 1 & 0 & 0 \\
 110 & 49 & 30 & 19 & 14 & 10 & 7 & 5 & 4 & 3 & 2 & 2 & 1 & 1 & 1 & 0 \\
 137 & 61 & 36 & 24 & 17 & 12 & 9 & 7 & 5 & 4 & 3 & 2 & 2 & 1 & 1 & 1 \\
 \hline
\end{array}
}
\end{equation*}
\end{center} 
\subcaption*{(i) $s_{n,k}$} 
\end{minipage} 

\begin{minipage}{\linewidth} 
\begin{center} 
\tiny 
\begin{equation*} 
\boxed{ 
\begin{array}{cccccccccccccccccc} \hline 
 1 & 0 & 0 & 0 & 0 & 0 & 0 & 0 & 0 & 0 & 0 & 0 & 0 & 0 & 0 & 0 \\
 -2 & 1 & 0 & 0 & 0 & 0 & 0 & 0 & 0 & 0 & 0 & 0 & 0 & 0 & 0 & 0 \\
 -1 & -1 & 1 & 0 & 0 & 0 & 0 & 0 & 0 & 0 & 0 & 0 & 0 & 0 & 0 & 0 \\
 0 & -1 & -1 & 1 & 0 & 0 & 0 & 0 & 0 & 0 & 0 & 0 & 0 & 0 & 0 & 0 \\
 0 & -1 & 0 & -1 & 1 & 0 & 0 & 0 & 0 & 0 & 0 & 0 & 0 & 0 & 0 & 0 \\
 1 & 1 & -2 & 0 & -1 & 1 & 0 & 0 & 0 & 0 & 0 & 0 & 0 & 0 & 0 & 0 \\
 0 & -1 & 1 & -1 & 0 & -1 & 1 & 0 & 0 & 0 & 0 & 0 & 0 & 0 & 0 & 0 \\
 0 & 1 & 0 & 0 & -1 & 0 & -1 & 1 & 0 & 0 & 0 & 0 & 0 & 0 & 0 & 0 \\
 2 & 0 & 0 & -1 & 1 & -1 & 0 & -1 & 1 & 0 & 0 & 0 & 0 & 0 & 0 & 0 \\
 -1 & 2 & -1 & 2 & -2 & 1 & -1 & 0 & -1 & 1 & 0 & 0 & 0 & 0 & 0 & 0 \\
 1 & -2 & 2 & -1 & 1 & -1 & 1 & -1 & 0 & -1 & 1 & 0 & 0 & 0 & 0 & 0 \\
 -1 & 2 & 0 & 1 & 0 & 0 & -1 & 1 & -1 & 0 & -1 & 1 & 0 & 0 & 0 & 0 \\
 2 & -2 & 2 & -2 & 2 & -1 & 1 & -1 & 1 & -1 & 0 & -1 & 1 & 0 & 0 & 0 \\
 -2 & 3 & -3 & 3 & -2 & 3 & -2 & 1 & -1 & 1 & -1 & 0 & -1 & 1 & 0 & 0 \\
 3 & -1 & 2 & -1 & 1 & -2 & 2 & -1 & 1 & -1 & 1 & -1 & 0 & -1 & 1 & 0 \\
 -3 & 2 & -2 & 2 & -1 & 2 & -1 & 1 & -1 & 1 & -1 & 1 & -1 & 0 & -1 & 1 \\
 \hline 
\end{array}
} 
\end{equation*} 
\end{center} 
\subcaption*{(ii) $s_{n,k}^{(-1)}$} 
\end{minipage}

\caption{The factorization where $C(q) := (q; q^2)_{\infty}^{-1}$} 
\label{figure_factpair_v4} 

\end{figure}


\begin{thebibliography}{10} 

\bibitem{HARDYANDWRIGHT}
G.~H. Hardy and E.~M. Wright.
\newblock {\it An Introduction to the Theory of Numbers}.
\newblock Oxford University Press, 2008.

\bibitem{MERCA-CIRCND} 
M. Merca, {Combinatorial interpretations of a recent convolution for the number of divisors of a positive integer}, 
{\it Journal of Number Theory}, 160, pp. 60--75 (2016).

\bibitem{MERCA-SCHMIDT1}
M. Merca and M. D. Schmidt, {Generating special arithmetic functions by 
  Lambert series factorizations}, \url{https://arxiv.org/abs/1706.00393} (2017). 
  (Under review at \textit{Experimental Mathematics}) 

\bibitem{MERCA-LSFACTTHM} 
M. Merca, {The {L}ambert series factorization theorem}, 
  {\it The Ramanujan Journal}, pp. 1--19 (2017). 

\bibitem{NISTHB}
F.~W.~J. Olver, D.~W. Lozier, R.~F. Boisvert, and C.~W. Clark.
\newblock {\it {NIST} Handbook of Mathematical Functions}.
\newblock Cambridge University Press, 2010.

\bibitem{SCHMIDT-COMBSUMSBDDDIV} 
M. D. Schmidt, {Combinatorial sums and identities involving generalized 
  divisor functions with bounded divisors}, 2017, 
  \url{https://arxiv.org/abs/1704.05595}. 
  (Under review at the \textit{Journal of Integer Sequences})

\bibitem{SCHMIDT-LSFACTTHM} 
M. D. Schmidt, {New recurrence relations and matrix equations for arithmetic 
  functions generated by {L}ambert series}, 2017, 
  \url{https://arxiv.org/abs/1701.06257}. 
  (Under review at \textit{Acta Arithmetica})
  
\bibitem{OEIS}
N. J. A. Sloane, {The Online Encyclopedia of Integer Sequences}, 2017, 
  \url{https://oeis.org/}. 

\end{thebibliography}
\end{document}